\documentclass[reqno]{amsart}

\usepackage{amssymb}
\usepackage{amsmath,amsthm}
\usepackage{amsmath,amscd}
\usepackage{amsfonts}
\usepackage{pstricks}
\usepackage{epsf}
\usepackage{xypic}
\usepackage{graphicx}
\usepackage{amsmath}
\usepackage[noadjust]{cite}
\newtheorem{theorem}{Theorem}[section]
\newtheorem{lemma}[theorem]{Lemma}
\newtheorem{corollary}[theorem]{Corollary}
\newtheorem{proposition}[theorem]{Proposition}
\theoremstyle{definition}
\newtheorem{definition}[theorem]{Definition}
\newtheorem{example}[theorem]{Example}

\theoremstyle{remark}
\newtheorem{remark}[theorem]{Remark}

\numberwithin{equation}{section}

\def\F{\mathcal{F}}

\def\H{\mathcal{H}}

\def\T{\mathcal{T}}
\def\P{\textbf{P}}
\def\Q{\textbf{Q}}

\def\E{\mathcal{E}}

\def\C{\mathcal{C}}
\def\X{\mathcal{X}}
\def\Y{\mathcal{Y}}

\def\Ext{\mbox{Ext}}
\def\Hom{\mbox{Hom}}
\def\End{\mbox{End}}

\def\Ker{\mbox{Ker}\hspace{.01in}}
\def\Im{\mbox{Im}\hspace{.01in}}
\begin{document}

\title{The representation invariants of 2-term silting complexes }

\author{Yonggang Hu}
\address{College of Applied Sciences, Beijing University of Technology, Beijing, P. R. China, 100124}
\curraddr{D\'{e}partement de math\'{e}matiques\\
Universit\'{e} de Sherbrooke\\ Sherbrooke, Qu\'{e}bec, Canada, J1K 2R1}
\email{huyonggang@emails.bjut.edu.cn\\huxy2601@usherbrooke.ca}
\thanks{This work was supported by CSC grant number 201906540017, from the China Scholarship Council}



\date{\today.}



\begin{abstract}
Let $A$ be a finite dimensional algebra over a field $k$ and $\P$ be a 2-term silting complex in $K^{b}(\text{proj}A)$. In this paper, we  investigate the representation dimension of $\text{End}_{D^{b}(A)}(\P)$ by using the silting theory. We show that if $\P$ is a separating silting complex with certain homological restriction, then rep.dim $A=$rep.dim $\text{End}_{D^{b}(A)}(\P)$. This gives  a proper generalization of the classical compare theorem of representation dimensions showed by Chen and Hu. It is well-known that $\text{H}^{0}(\P)$ is a tilting $A/\text{ann}_{A}(\P)$-module. We also show that rep.dim $\text{End}_{A}(\text{H}^{0}(\P))=$rep.dim $A/\text{ann}_{A}(\P)$ if $\P$ is a separating and splitting silting complex.
\end{abstract}

\maketitle

\section{Introduction}
The  concept of representation dimension was first introduced by Auslander \cite{Auslander} in 1971. It is an important homological invariant in the  representation theory. He proved that an artin algebra $A$ is representation-finite if and only if its  representation dimension at most two. It means that the representation dimension gives a reasonable way of  measuring how far an artin algebra is from being of finite representation  type. In 1998, Reiten asked whether any artin algebra has a finite  representation dimension. After this, Iyama \cite{Iyama} gave a positive answer to this question. Moreover, representation dimension of algebras is closely relative to other homological conjectures, such as the finitistic dimension conjecture. In particular, Igusa and Todorov \cite{Igusa} proved that if rep.dim$A\leq3$, then $A$ has finite finitistic  dimension. However, it was unsure that whether  the representation dimension of an artin algebra can be greater than three, until in 2005 Rouquier \cite{Rouquier}  showed that representation dimension of an artin algebra may be arbitrarily large and constructed examples of algebras with  representation dimension larger than or equal to four.\par
Up to now,  the representation dimensions of several important classes of algebras are known to be at most three, such as hereditary algebras \cite{Auslander}, torsionless-finite algebras \cite{Ringle01}, glued algebras \cite{Coelho}, tilted algebras \cite{Assem01}, quasi-tilted algebras \cite{Oppermann}, iterated tilted algebras \cite{Happel}, special biserial algebras \cite{Erdmann}, cluster-concealed algebras \cite{Chaio}, and so on. In general, for a given artin algebra, it is difficult to know and compute the actual value of its  representation dimension. However, there is a wise strategy to calculate the representation dimension by comparing two closely related algebras. Then by the representation dimension of a known algebra,  one can measure that of unknown algebra. Along this philosophy, it is nature to consider  comparing the   representation dimensions between algebra  $A$ and $\text{End}_{A}(T)$, where $T$ is a classical tilting right $A$-module. In this case, $T$ induces two torsion pairs ($\T(T)$, $\F(T)$), ($\X(T)$, $\Y(T)$) in mod$A$ and mod$\text{End}_{A}(T)$, respectively. These torsion pairs split the module categories into some different pieces. It turn out to be effective to compute the representation dimension by using torsion pairs, see \cite{Happel} and \cite{ChenHu}.\par
As a generalization of the classical tilting theory, the concept of silting complexes originated from Keller and Vossieck. In particular, Hoshino \cite{Hoshino} showed that  2-term silting complexes can induce torsion pairs in module categories.  More recently, Buan and Zhou \cite{Buan} gave a generalization of the classical tilting theorem, called silting theorem. They described the relations of torison pairs between mod$A$ and mod$B$ by using the natural equivalences induced by Hom and Ext functors, where $B=\text{End}_{D^{b}(A)}(\P)$ and $\P$ is a 2-term silting complex in $K^{b}(\text{proj}A)$. It provides us with a basic framework to research the representation dimension by the silting theory. For more important homological results on the 2-term silting complexes,  we refer the reader to \cite{Buan02} and \cite{Buan03}.\par
In this paper, we consider the representation invariants induced by 2-term silting complexes. In details, we focus on when the representation dimensions of $A$ and $\text{End}_{D^{b}(A)}(\P)$ are coincide. Now, we present one of our main results as follows.\par
\begin{theorem}\label{th1} Let $A$ be a finite dimensional algebra and  let $\P$ be a 2-term separating silting  complex in {\rm $K^{b}(\text{proj}A)$} such that {\rm Id$_{A}X\leq1$} for each $X\in\F(\P)$ and {\rm$B=\text{End}_{D^{b}(A)}(\P)$}. Then {\rm rep.dim$B=$rep.dim$A$}.
\end{theorem}
Applying this result into the classical tilting theory, we can obtain \cite[Theorem 3.1]{ChenHu}. Meanwhile, we give an example to  illustrate that it is a proper generalization, see Example \ref{ex4.3}.\par
On other hand, it is well-known that if $\P$ is a 2-term silting complex in {\rm $K^{b}(\text{proj}A)$}, then  $\text{H}^{0}(\P)$ is a tilting $A/\text{ann}_{A}(\P)$-module, where $\text{ann}_{A}(\P)$ is  the annihilator of $\text{H}^{0}(\P)$. In general, we know that $\text{End}_{A}(\text{H}^{0}(\P))$ is a factor algebra of $\text{End}_{A}(\P)$. It is interesting to consider to describe the relationship of the representation dimensions between $\text{End}_{A}(\text{H}^{0}(\P))$ and $A$. Thus, we have the following results.
\begin{proposition} Let $A$ be a finite dimensional algebra and $\P$ be a 2-term  splitting and separating silting  complex in {\rm $K^{b}(\text{proj}A)$}.  Then {\rm rep.dim$\End_{A}(\text{H}^{0}(\P))=$ rep.dim$A/\text{ann}_{A}(\P)$}.
\end{proposition}
The paper is organized as follows. In Section 2, we recall some well-known results on the silting theory and the representation dimension. In Section 3, we prove our main results. In Section 4, we provide some examples to illustrate that  anyone of the conditions of Theorem \ref{th1} cannot be removed. In Section 5, we compare the representation dimensions of $\End_{A}(\text{H}^{0}(\P))$ and $A/\text{ann}_{A}(\P)$.

\section{Preliminaries}
Let $A$ be a finite dimensional $k$-algebra with $k$ is a field. We denote by mod$A$ the category of finitely generated right $A$-modules. Let $D(-)=\Hom_{k}(-,k)$ be the $k$-duality. We denote by proj$A$ the full subcategory of mod$A$ generated by the projective modules. Let $D^{b}(A)$ be the bounded derived category, with shift functor $\Sigma$ and $K^{b}(\text{proj}A)$ the bounded homotopy category of finitely generated projective  right $A$-modules.\par
A complex $\textbf{P}$ is said to be 2-term if $P^{i}=0$ for $i\neq0, 1$. Recall that a 2-term complex $\P$ in $K^{b}(\text{proj}A)$ is said to be silting if it satisfies the following two conditions
\begin{enumerate}
  \item  $\Hom_{K^{b}(\text{proj}A)}(\P,\Sigma\P)$=0;
  \item  thick$\P$=$K^{b}(\text{proj}A)$ where thick$\P$ is the smallest triangulated subcategory closed under direct summands containing $\P$.
\end{enumerate}
In addition, if $\P$ satisfies  $\Hom_{K^{b}(\text{proj}A)}(\P,\Sigma^{-1}\P)$=0, then $\P$ is said to be tilting.\par
Let $\P$ be a 2-term silting complex in {\rm $K^{b}(\text{proj}A)$}, and consider the following two full subcategories of mod$A$
{\rm\begin{align*}
  \T(\P) & =\{~X\in \text{mod}A~|~\Hom_{D^{b}(A)}(\P,\Sigma X)=0\} \\
  \F(\P) & =\{~Y\in \text{mod}A~|~\Hom_{D^{b}(A)}(\P,Y)=0\}.
\end{align*}}
\begin{theorem}{\rm \cite{Buan}}\label{th1-1} Let $\P$ be a 2-term silting complex in {\rm $K^{b}(\text{proj}A)$}, and let {\rm$B=\text{End}_{D^{b}(A)}(\P)$}.
\begin{enumerate}
\item $\C(\P)$ is an abelian category and the short exact sequences in $\C(\P)$ are precisely the triangles in $D^{b}(A)$ all of whose vertices are objects in $\C(\P)$.

  \item The pair {\rm($\T(\P)$, $\F(\P)$)} is a torsion pair in {\rm mod$A$}. The pair {\rm($\Sigma\F(\P)$, $\T(\P)$)} is a torsion pair in  $\C(\P)$.
  \item {\rm$\Hom_{D^{b}(A)}(\P,-):\C(\P)\rightarrow \text{mod}B$ } is an equivalence of abelian categories.
  \item There is a triangle
  \begin{equation*}
    A\rightarrow \P'\xrightarrow{f}\P''\rightarrow \Sigma A
  \end{equation*}
  with $\P'$, $\P''$ in {\rm add$\P$}.\\
  Consider the 2-term complex $\Q$ in {\rm $K^{b}(\text{proj}B)$} induced by the map
   {\rm\begin{equation*}
    \Hom_{D^{b}(A)}(\P, f): \Hom_{D^{b}(A)}(\P, \P')\xrightarrow{}\Hom_{D^{b}(A)}(\P, \P'').
  \end{equation*}}
  \item $\Q$ is a 2-term silting complex in {\rm $K^{b}(\text{proj}B)$} such that {\rm\begin{align*}
                        \T(\Q) &=\X(\P)=\Hom_{D^{b}(A)}(\P,\Sigma\F(\P)) \\
                        \F(\Q) & =\Y(\P)=\Hom_{D^{b}(A)}(\P,\T(\P)).
                      \end{align*}}
  \item There is an algebra epimorphism {\rm$\Phi_{\P}:A\rightarrow \overline{A}=\text{End}_{D^{b}(B)}(\Q)$}.
  \item $\Phi_{\P}$ is an isomorphism if and only if $\P$ is tilting.
  \item Let {\rm$\Phi_{\ast}: \text{mod}\overline{A}\rightarrow \text{mod}A$} be the inclusion functor. Then one obtains the following picture,
which also shows the quasi-inverse equivalences between the pair {\rm($\T(\P)$,~$\F(\P)$)} and {\rm($\T(\Q)$,~$\F(\Q)$)}.
\begin{figure}[h]
  \centering
  \includegraphics[height=4.8cm,width=8cm]{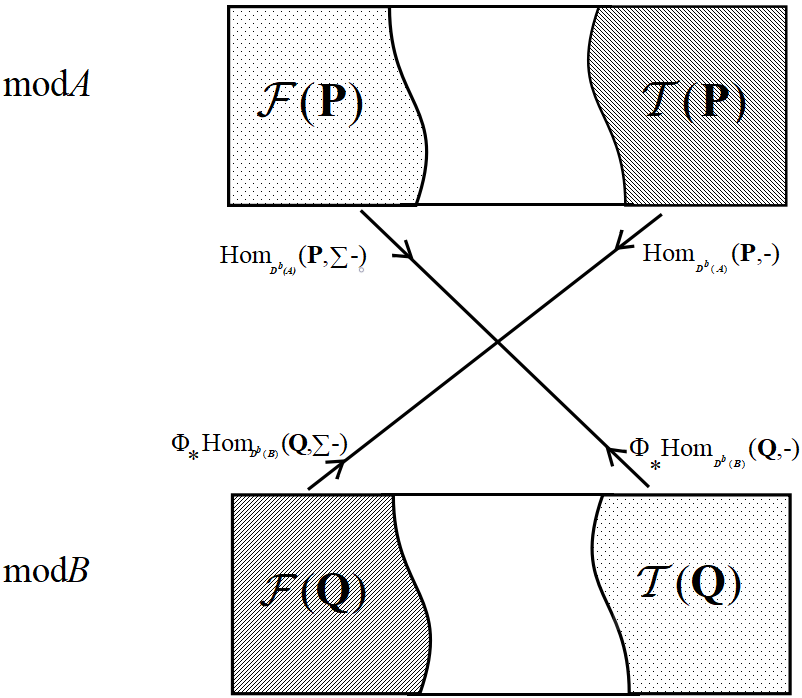}\\
\end{figure}

\end{enumerate}
\end{theorem}
In what following, the symbol $\Q$ always denotes the induced complex $\Q$. It is a  $2$-term silting complex in $K^{b}(\text{proj}B)$ such that the induced pair $(\T(\Q),\F(\Q))=(\X(\P),\Y(\P))$.
\begin{definition} {\rm \cite{Buan}} Let $\P$ is a 2-term silting complex in $K^{b}(\text{proj}A)$.
\begin{enumerate}
  \item $\P$ is called splitting if the induced torsion pair ($\X(\P)$, $\Y(\P)$) in mod$B$ is split.
  \item $\P$ is called separating if the induced torsion pair ($\T(\P)$, $\F(\P)$) in mod$A$ is split.
\end{enumerate}
 \end{definition}
 \begin{remark} \cite{Buan} Let $\P$ be a 2-term silting complex in $K^{b}(\text{proj}B)$.
 \begin{enumerate}
                  \item $\P$ is splitting if and only if $\Ext^{2}_{A}(\T(\P),\F(\P))=0$.
                  \item If $\P$ is separating, then $\P$ is a tilting complex.
                  \item Suppose that $\P$ is both splitting and separating. In this case,  $\Q$ is a separating silting complex and so, it is a tilting complex. Then,  $A\cong \text{End}_{D^{b}(B)}(\Q)$. In this case,  $(\X(\Q),\Y(\Q))=(\T(\P),\F(\P))$. Then $\Q$ is also a splitting  silting complex.
                \end{enumerate}

 \end{remark}
 \begin{lemma}{\rm \cite{Buan}} \label{lemma2.14}Let $\P$ be a 2-term silting  complex in {\rm $K^{b}(\text{proj}A)$} and {\rm($\T(\P)$, $\F(\P)$)} be the induced torsion pair in {\rm mod$A$}. Then the following hold.
\begin{enumerate}
  \item For any $X\in$ {\rm mod$A$}, $X\in$ {\rm add$H^{0}(\P)$} if and only if $X$ is {\rm Ext}-projective in $\T(\P)$.
  \item For any $X\in$ {\rm mod$A$}, $X\in$ {\rm add$H^{-1}(\nu\P)$} if and only if $X$ is {\rm Ext}-injective in $\F(\P)$.
\end{enumerate}
\end{lemma}
Let $M$ be a module in mod$A$. $M$ is said to be a generator of mod$A$ if $A\in$ add$M$. Dually, one can define the  cogenerator. The representation of algebra $A$ is defined as
\begin{center}
  rep.dim$A=$inf$\{~\text{gl.dim}(\text{End}_{A}(M))~| ~M~\text{a generator and cogenerator of mod}A\}.$
\end{center}
\begin{lemma}{\rm(\cite{Auslander},\cite[Lemma 2.1]{Erdmann})}\label{lemma2.0} Let $A$ be an algebra, $n$ be a non-negative integer at least {\rm$2$} and $M$ be a generator-cogenerator for {\rm mod$A$}. Then the following statements are equivalent:\begin{enumerate}
                            \item {\rm gl.dim($\text{End}_{A}(M))\leq n+2$},
                            \item For each $A$-module $X$, there exists an exact sequence
$$0\rightarrow M_{n}\rightarrow M_{n-1}\rightarrow\cdots\rightarrow M_{0}\longrightarrow X\rightarrow0$$
with $M_{i}$ in {\rm add$M$} for all $i$, such that the induced sequence
{\rm$$0\rightarrow\text{Hom}_{A}(M, M_{n})\rightarrow\cdots\rightarrow\text{Hom}_{A}(M, M_{1})\rightarrow \text{Hom}_{A}(M,X)\rightarrow0$$}
is exact.
\item For each $A$-module $X$, there exists an exact sequence
$$0\rightarrow X\rightarrow M_{0}\rightarrow\cdots\rightarrow M_{1}\rightarrow M_{n}\rightarrow0$$
with $M_{i}$ in {\rm add$M$} for all $i$, such that the induced sequence
{\rm$$0\rightarrow\text{Hom}_{A}( M_{n},M )\rightarrow\cdots\rightarrow\text{Hom}_{A}( M_{0},M)\rightarrow \text{Hom}_{A}(X,M)\rightarrow0$$}
\end{enumerate}
\end{lemma}
An $A$-module $M$ is said to be  an Aulsander generator of mod$A$ if it satisfies that gl.dim($\text{End}_{A}(M)$)=rep.dim$A$.
\section{Main result}
In this section, we will compare the the representation dimensions of $A$ and $\text{End}_{D^{b}(A)}(\P)$.\par
The following  result was proved in \cite{Hoshino}, in the setting of abelian categories with arbitrary coproducts. Indeed, it is also true in our case. The proof of the following lemma has contained in \cite{Buan}. For convenience, we provide the details of proof here.
\begin{lemma}\label{lem1-1} For any $X\in$ {\rm mod}$A$, {\rm$\Hom_{D^{b}(A)}(\P, \Sigma^{i}X)=0$} for any $i<0$ and $i>1$.
\end{lemma}
\begin{proof} It is easy to check that $\Hom_{D^{b}(A)}(\P, \Sigma^{i}X)=0$ for any $i>1$. Now we prove the former case. Assume that $\P:P^{-1}\xrightarrow{d}P^{0}$ where all $P^{i}$ are finitely generated projective modules. Then there is a distinguished triangle
\begin{equation}\label{eq1-1}
 P^{-1}\xrightarrow{-d}P^{0}\rightarrow \P\rightarrow \Sigma P^{-1}.
\end{equation}
Applying the functor $\Hom_{D^{b}(A)}(-, \Sigma^{i}X)=0$ to the sequence \ref{eq1-1}, we have the following sequence
$$\cdots \rightarrow \Hom_{D^{b}(A)}(\Sigma P^{-1}, \Sigma^{i}X)\rightarrow \Hom_{D^{b}(A)}(\P, \Sigma^{i}X)\rightarrow \Hom_{D^{b}(A)}( P^{0}, \Sigma^{i}X)\rightarrow\cdots$$
Note that for any $i<0$, $ \Hom_{D^{b}(A)}(\Sigma P^{-1}, \Sigma^{i}X)=\Hom_{D^{b}(A)}( P^{0}, \Sigma^{i}X)=0$. Therefore, $\Hom_{D^{b}(A)}(\P, \Sigma^{i}X)=0$ for any $i<0$.
\end{proof}
Next, we shall character the right $B$-module $\text{End}_{A}(\text{H}^{0}(\P))$.
\begin{lemma}\label{lem1-4} Let $\P$ be a 2-term splitting silting complex in {\rm $K^{b}(\text{proj}A)$} and {\rm$B=\text{End}_{D^{b}(A)}(\P)$}. Then the right $B$-module {\rm$\text{End}_{A}(\text{H}^{0}(\P))$} is projective. In particular, {\rm$B\cong \Hom_{D^{b}(A)}(\P,\text{H}^{0}(\P))\oplus \Hom_{D^{b}(A)}(\P,\Sigma\text{H}^{-1}(\textbf{P}))$}.
\end{lemma}
\begin{proof} Let $X$ be an indecomposable right $B$-module. Since $\P$ is a splitting silting complex, $X\in \X(\P)$ or $X\in \Y(\P)$. Note that $\text{End}_{A}(\text{H}^{0}(\P))\cong \Hom_{D^{b}(A)}(\P,\text{H}^{0}(\P))$ which is in $ \Y(\P)$ as right $B$-module since $\text{H}^{0}(\P)\in\T(\P)$. If $X\in \X(\P)$, then,  we get the isomorphism $\Ext^{1}_{B}(\text{End}_{A}(\text{H}^{0}(\P)),X)\cong D\underline{\Hom}(\tau^{-1}X, \text{End}_{A}(\text{H}^{0}(\P))$ by AR-formula. Since  ($\X(\P)$, $\Y(\P)$) is split and $X\in \X(\P)$, $\tau^{-1}X\in \X(\P)$ and hence, $\Ext^{1}_{B}(\text{End}_{A}(\text{H}^{0}(\P),X)=0$. If $X\in \Y(\P)$, then there exists a right $A$-module $X'\in \T(\P)$ such that $X\cong \Hom_{D^{b}(A)}(\P,X')$. For any short exact sequence
\begin{equation}\label{eq1-5}
  0\rightarrow \Hom_{D^{b}(A)}(\P,X')\rightarrow U\rightarrow \Hom_{D^{b}(A)}(\P,\text{H}^{0}(\P))\rightarrow0
\end{equation}
in $\Ext^{1}_{B}(\text{End}_{A}(\text{H}^{0}(\P),X)$, $U\in \Y(\P)$ since the first and third terms are in $\Y(\P)$ and there exists $U'\in \T(\P)$ such that $U\cong \Hom_{D^{b}(A)}(\P,U')$. It follows that there is an exact sequence
\begin{equation}\label{eq1-6}
  0\rightarrow X'\rightarrow U'\rightarrow \text{H}^{0}(\P)\rightarrow0
\end{equation}
in mod$A$. Note that $\text{H}^{0}(\P)$ is an $\Ext$-projective module in $\T(\P)$. Thus, the sequence (\ref{eq1-6}) splits. It yields that the sequence (\ref{eq1-5}) splits. Hence, we have that $\Hom_{D^{b}(A)}(\P,\text{H}^{0}(\P))$ is projective in mod$B$.\par
Note that for any 2-term complex $\textbf{Y}$ in $D^{b}(A)$, there is a triangle
$$\Sigma\text{H}^{-1}(\textbf{Y})\rightarrow \textbf{Y}\rightarrow \text{H}^{0}(\textbf{Y})\rightarrow \Sigma^{2}\text{H}^{-1}(\textbf{Y}).$$
Applying $\Hom_{D^{b}(A)}(\P,-)$ to the triangle
$$\Sigma\text{H}^{-1}(\textbf{P})\rightarrow \textbf{P}\rightarrow \text{H}^{0}(\textbf{P})\rightarrow \Sigma^{2}\text{H}^{-1}(\textbf{P}),$$
we get the following long exact sequence of right $B$-modules
\begin{align*}
  \cdots\rightarrow &\Hom_{D^{b}(A)}(\P,\Sigma^{-1}\text{H}^{0}(\textbf{P}))\rightarrow\Hom_{D^{b}(A)}(\P,\Sigma\text{H}^{-1}(\textbf{P}))\rightarrow \Hom_{D^{b}(A)}(\P,\textbf{P}) \\
  \rightarrow & \Hom_{D^{b}(A)}(\P,\text{H}^{0}(\textbf{P}))\rightarrow\Hom_{D^{b}(A)}(\P, \Sigma^{2}\text{H}^{-1}(\textbf{P}))\rightarrow\cdots.
\end{align*}
By Lemma \ref{lem1-1}, we get a short exact sequence
\begin{align*}
  0\rightarrow\Hom_{D^{b}(A)}(\P,\Sigma\text{H}^{-1}(\textbf{P}))\rightarrow \Hom_{D^{b}(A)}(\P,\textbf{P})
  \rightarrow \Hom_{D^{b}(A)}(\P,\text{H}^{0}(\textbf{P}))\rightarrow0.
\end{align*}
The result follows form that $\Hom_{D^{b}(A)}(\P,\text{H}^{0}(\textbf{P}))$ is projective.
\end{proof}

\begin{lemma}\label{lemm2-10}  Let $\P$ be a 2-term  tilting complex in {\rm $K^{b}(\text{proj}A)$}. Then $\nu\P\in\C(\P)$. In this case, {\rm$\text{H}^{0}(\nu\P)\in\T(\P)$} and {\rm$\text{H}^{-1}(\nu\P)\in\F(\P)$}.
\end{lemma}
\begin{proof}
It suffices to show that $\nu\P\in \text{D}^{\leq0}(\P)\cap\text{D}^{\geq0}(\P)$. The result follows from  the following equations
\begin{align*}
  \Hom_{D^{b}(A)}(\P,\Sigma^{i}\nu\P) &= \text{H}^{i}\Hom^{\bullet}(\P,\nu\P) \\
  & \cong \text{H}^{i}D(\P^{\ast}\otimes_{A}^{\bullet}\P)\\
  &\cong \text{H}^{i}D\Hom^{\bullet}(\P,\P)\\
  &\cong D\text{H}^{-i}\Hom^{\bullet}(\P,\P)\\
  &\cong D\Hom_{D^{b}(A)}(\P,\Sigma^{-i}\P).
\end{align*}
The rest results are from \cite[Lemma 2.13]{Abe}.
\end{proof}
Dually, we can describe the right $B$-module {\rm$\Hom_{D^{b}(A)}(\P,\Sigma\text{H}^{-1}(\nu\P))$}.
\begin{lemma}\label{lemm2-11} Let $\P$ be a 2-term  splitting silting complex in {\rm $K^{b}(\text{proj}A)$} and {\rm$B=\text{End}_{D^{b}(A)}(\P)$}. Then the right $B$-module {\rm$\Hom_{D^{b}(A)}(\P,\Sigma\text{H}^{-1}(\nu\P))$} is an injective module. In this case, {\rm $DB\cong \Hom_{D^{b}(A)}(\P,\Sigma \text{H}^{-1}(\nu\P))\oplus\Hom_{D^{b}(A)}(\P, \text{H}^{0}(\nu\P))$}.
\end{lemma}
\begin{proof} It is easy to check that $DB=\Hom_{D^{b}(A)}(\P,\nu\P)$. Since $\text{H}^{-1}(\nu\P)\in\F(\P)$, we know that $\Hom_{D^{b}(A)}(\P,\Sigma\text{H}^{-1}(\nu\P))\in\X(\P)$. If $Y\in \X(\P)$, then,  we get the isomorphism $$\Ext^{1}_{B}(Y ,\Hom_{D^{b}(A)}(\P,\Sigma\text{H}^{-1}(\nu\P)))\cong D\overline{\Hom}_{B}(\Hom_{D^{b}(A)}(\P,\Sigma\text{H}^{-1}(\nu\P)),\tau Y)$$ by AR-formula. Since  ($\X(\P)$, $\Y(\P)$) is split and $Y\in \Y(\P)$, $\tau Y\in \Y(\P)$ and hence, $\Ext^{1}_{B}(\Y ,\Hom_{D^{b}(A)}(\P,\Sigma \text{H}^{-1}(\nu\P)))=0$. If $Y\in \X(\P)$, then there exists a right $A$-module $Y'\in \F(\P)$ such that $X\cong \Hom_{D^{b}(A)}(\P,\Sigma Y')$. For any short exact sequence
\begin{equation}\label{eq1-7}
  0\rightarrow \Hom_{D^{b}(A)}(\P,\Sigma\text{H}^{-1}(\nu\P))\rightarrow V\rightarrow \Hom_{D^{b}(A)}(\P,\Sigma Y')\rightarrow0
\end{equation}
in $\Ext^{1}_{B}(Y, \Hom_{D^{b}(A)}(\P,\Sigma \text{H}^{-1}(\nu\P)))$, $V\in \X(\P)$ since the first and third terms are in $\X(\P)$ and there exists $V'\in \F(\P)$ such that $V\cong \Hom_{D^{b}(A)}(\P,\Sigma V')$. It follows that there is an exact sequence
\begin{equation}\label{eq1-8}
  0\rightarrow \text{H}^{-1}(\nu\P)\rightarrow V'\rightarrow Y'\rightarrow0
\end{equation}
in $\F(\P)$. Note that $\text{H}^{-1}(\nu\P)$ is an $\Ext$-injective module in $\F(\P)$. Thus, the sequence (\ref{eq1-8}) splits. It yields that the sequence (\ref{eq1-7}) splits. Thus, we know that $\Hom_{D^{b}(A)}(\P,\Sigma \text{H}^{-1}(\nu\P))$ is injective in mod$B$.\par
Applying $\Hom_{D^{b}(A)}(\P,-)$ to the triangle
$$\Sigma\text{H}^{-1}(\nu\textbf{P})\rightarrow \nu\textbf{P}\rightarrow \text{H}^{0}(\nu\textbf{P})\rightarrow \Sigma^{2}\text{H}^{-1}(\nu\textbf{P}),$$
we get the following long exact sequence of right $B$-modules
\begin{align*}
  \cdots\rightarrow &\Hom_{D^{b}(A)}(\P,\Sigma^{-1}\text{H}^{0}(\nu\textbf{P}))\rightarrow\Hom_{D^{b}(A)}(\P,\Sigma\text{H}^{-1}(\nu\textbf{P}))\rightarrow \Hom_{D^{b}(A)}(\P,\nu\textbf{P}) \\
  \rightarrow & \Hom_{D^{b}(A)}(\P,\text{H}^{0}(\nu\textbf{P}))\rightarrow\Hom_{D^{b}(A)}(\P, \Sigma^{2}\text{H}^{-1}(\nu\textbf{P}))\rightarrow\cdots.
\end{align*}
By Lemma \ref{lem1-1}, we get a short exact sequence
\begin{align*}
  0\rightarrow\Hom_{D^{b}(A)}(\P,\Sigma\text{H}^{-1}(\nu\textbf{P}))\rightarrow \Hom_{D^{b}(A)}(\P,\nu\textbf{P})
  \rightarrow \Hom_{D^{b}(A)}(\P,\text{H}^{0}(\nu\textbf{P}))\rightarrow0.
\end{align*}
The result follows form that $\Hom_{D^{b}(A)}(\P,\text{H}^{-1}(\nu\textbf{P}))$ is injective.
\end{proof}
\begin{proposition}\label{prop2.12} Let $A$ be a finite dimensional algebra and  let $\P$ be a 2-term silting  complex in {\rm $K^{b}(\text{proj}A)$} such that {\rm Id$_{A}X\leq1$} for each $X\in\F(\P)$ and {\rm$B=\text{End}_{D^{b}(A)}(\P)$}. Then $\P$ is separating if and only if  for any $M\in \X(\P)$, {\rm Pd$_{B}M\leq1$}.
\end{proposition}
\begin{proof}
For the necessity, it is enough to prove that Pd$_{B}M\leq1$, for any module $M\in\T(\Q)$.\par
By the assumption, we know that $\P$ is a splitting silting complex. Since $\P$ is a separating silting complex, then $\P$ is a tilting complex and hence, $\Q$ is a tilting complex. Then $A\cong \text{End}_{D^{b}(B)}(\Q)$. Moreover, the torsion pairs ($\T(\P)$, $\F(\P)$) and ($\T(\Q)$, $\F(\Q)$) are split in mod$A$ and mod$B$, respectively.\par
Since ($\T(\Q)$, $\F(\Q)$) is split, for any indecomposable $N\in$ mod$B$, $N\in\T(\Q)$ or $N\in\F(\Q)$. If $N\in\F(\Q)$, then
{\rm\begin{align*}
  \Ext^{2}_{B}(M,N) &\cong \Hom_{D^{b}(B)}(M,\Sigma^{2}N) \\
   &= \Hom_{D^{b}(B)}(M,\Sigma (\Sigma N))~~~~(\text{by Theorem \ref{th1-1}(1)})\\
   &\cong\Ext^{1}_{A}(\Hom_{D^{b}(B)}(\Q,M),\Hom_{D^{b}(B)}(\Q,\Sigma N))~~~~(\text{by Theorem \ref{th1-1}(2)}).
\end{align*}}
Note that $\Hom_{D^{b}(B)}(\Q,\Sigma N)\in \T(\P)$ and $\Hom_{D^{b}(B)}(\Q,M)\in\F(\P)$. Since ($\T(\P)$, $\F(\P)$) is split, $\Ext^{1}_{A}(\Hom_{D^{b}(B)}(\Q,M),\Hom_{D^{b}(B)}(\Q,\Sigma N))=0$. Thus, $\Ext^{2}_{B}(M,N)=0$. If $N\in \T(\Q)$, then
{\rm\begin{align*}
  \Ext^{2}_{B}(M,N) &\cong \Hom_{D^{b}(B)}(M,\Sigma^{2}N) \\
   &\cong\Hom_{D^{b}(A)}(\Hom_{D^{b}(B)}(\Q,M),\Sigma^{2}\Hom_{D^{b}(B)}(\Q, N))\\
   &\cong\Ext^{2}_{A}(\Hom_{D^{b}(B)}(\Q,M),\Hom_{D^{b}(B)}(\Q, N)).
\end{align*}}
Since $\Hom_{D^{b}(B)}(\Q, N)\in\F(\P)$,  by the assumption on $\F(\P)$, we have that $$\Ext^{2}_{A}(\Hom_{D^{b}(B)}(\Q,M),\Hom_{D^{b}(B)}(\Q, N))=0$$ and hence $\Ext^{2}_{B}(M,N)=0$. Then the claim holds.\par
For the sufficiency, assume that $X\in\F(\P)$ and $Y\in\T(\P)$. It suffices to show that $\Ext_{A}^{1}(X,Y)=0$. Indeed, we have the following isomorphisms
\begin{align*}
  \Ext_{A}^{1}(X,Y) & \cong \Hom_{D^{b}(A)}( X,\Sigma Y)\\
  & \cong\Hom_{D^{b}(A)}( \Sigma X,\Sigma^{2} Y)\\
  & \cong\Hom_{D^{b}(B)}( \Hom_{D^{b}(A)}(\P,\Sigma X),\Sigma^{2}  \Hom_{D^{b}(A)}(\P,Y))\\
  &\cong\Ext_{B}^{2}(\Hom_{D^{b}(A)}(\P,\Sigma X), \Hom_{D^{b}(A)}(\P,Y)).
\end{align*}
Note that $\Hom_{D^{b}(A)}(\P,\Sigma X)\in\X(\P)$. By the assumption, we know that the projective dimension of  $\Hom_{D^{b}(A)}(\P,\Sigma X)$ at most 1. It implies that $\Ext_{A}^{1}(X,Y)=0$.
\end{proof}
Dually, one can prove the following result.
\begin{proposition}\label{proposition 2.10} Let $A$ be a finite dimensional algebra and  let $\P$ be a 2-term silting  complex in {\rm $K^{b}(\text{proj}A)$} such that {\rm Pd$_{A}X\leq1$} for each $X\in\T(\P)$ and {\rm$B=\text{End}_{D^{b}(A)}(\P)$}. Then $\P$ is separating if and only if  for any $N\in \Y(\P)$, {\rm id$_{B}N\leq1$}.
\end{proposition}

In what follows, for convenience, we denote by $\H(-)$ the functor $\Hom_{D^{b}(A)}(\P,-)$ and by $\E(-)$ the functor $\Hom_{D^{b}(A)}(\P,\Sigma-)$.
\begin{lemma} \label{lemma2.13} Let $A$ be a finite dimensional algebra and  let $\P$ be a 2-term separating silting  complex in {\rm $K^{b}(\text{proj}A)$} such that {\rm Id$_{A}X\leq1$} for each $X\in\F(\P)$ and {\rm$B=\text{End}_{D^{b}(A)}(\P)$}.  Suppose that $M$ is a generator of {\rm mod$A$} and {\rm$N=B\oplus\H(M)\oplus Y$} with $Y\in\X(\P)$. Then
{\rm$
  \text{Pd}_{\text{End}_{B}(N)}\Hom_{B}(N,\H(U))\leq \text{Pd}_{\text{End}_{A}(M)}\Hom_{B}(M,U)
$}
for any $U\in\T(\P)$.
\end{lemma}
\begin{proof} If $\text{Pd}_{\text{End}_{A}(M)}\Hom_{B}(M,U)=\infty$, then the result holds. Otherwise, we assume that $\text{Pd}_{\text{End}_{A}(M)}\Hom_{B}(M,U)=n<\infty$. Since $M$ is a generator of {\rm mod$A$}, there is an add$M$-resolution of $U$, that is, there is a long exact sequence
$$0\rightarrow M_{n}\xrightarrow{f_{n}} \cdots\rightarrow M_{1}\xrightarrow{f_{1}} M_{0}\xrightarrow{f_{0}} U\rightarrow0$$
such that it keeps exact after applying $\Hom_{A}(M',-)$ with $M'\in$ add$M$.  Then there are a family of short each short exact sequences $$0\rightarrow K_{i+1}\xrightarrow{\eta_{i}} M_{i}\xrightarrow{\overline{f_{i}}} K_{i}\rightarrow0$$ for $i=0,1,\cdots,n-1$, where $K_{0}=U$, $K_{n}=M_{n}$, $K_{j}=\Ker f_{j}$ for $j=1,2,\cdots,n-1$ and $\overline{f_{i}}$ are right add$M$-approximations.
Applying the functor $\H$ to these sequences, we obtain a long exact sequence
$$0\rightarrow \H(K_{i+1})\xrightarrow{\H(\eta_{i})} \H(M_{i})\xrightarrow{\H(\overline{f_{i}})} \H(K_{i})\rightarrow\E(K_{i+1})\rightarrow\E(M_{i})\rightarrow \E(K_{i})\rightarrow0$$
Set $\Omega_{i}=\text{Coker}\H(\overline{f_{i}})$. Then we get a long exact sequence
$$0\rightarrow\Omega_{i}\rightarrow\E(K_{i+1})\rightarrow\E(M_{i})\rightarrow \E(K_{i})\rightarrow0$$
Note that $\E(K_{i+1}),\E(M_{i}), \E(K_{i})$ are in $\X(\P)$. Then by Proposition \ref{prop2.12}, we know that the projective dimensions of $\E(K_{i+1}),\E(M_{i}), \E(K_{i})$ are at most 1. Thus, Pd$_{B}\Omega_{i}\leq1$. Assume that $0\rightarrow B_{i}^{1}\xrightarrow{\beta_{i}} B_{i}^{0}\xrightarrow{\alpha_{i}} \Omega_{i}\rightarrow0$ is the projective resolution of $\Omega_{i}$ with $B_i^{j}\in$ add$B$. Then we obtain the following commutative diagram
$$\xymatrix{&0\ar[d]&0\ar[d]& 0\ar[d]& \\
0\ar[r]&\H(K_{i+1})\ar[d]_{\H(\eta_{i})}\ar[r]^-{\left[
              \begin{smallmatrix}
                1\\0
              \end{smallmatrix}
            \right]}&\H(K_{i+1})\oplus B_{i}^{1}\ar[d]^-{\left[
              \begin{smallmatrix}
                \H(\eta_{i})&0\\
                0&\beta_{i}
              \end{smallmatrix}
            \right]}\ar[r]^-{\left[
              \begin{smallmatrix}
                0~1
              \end{smallmatrix}
            \right]}&B_{i}^{1}\ar[r]\ar[d]^{\beta_{i}}&0  \\
           0\ar[r] &\H(M_{i})\ar[d]_{\overline{\H(\overline{f_{i}})}}\ar[r]^-{\left[
              \begin{smallmatrix}
                1\\0
              \end{smallmatrix}
            \right]}&\H(M_{i})\oplus B_{i}^{0}\ar[d]^-{\left[
              \begin{smallmatrix}
                \H(\overline{f_{i}})&\mu_{i}
              \end{smallmatrix}
            \right]}\ar[r]^-{\left[
              \begin{smallmatrix}
                0~1
              \end{smallmatrix}
            \right]}&B_{i}^{0}\ar[r]\ar[d]^{\alpha_{i}}& 0  \\
           0\ar[r]  &\Im\H(\overline{f_{i}}) \ar[r]^{\eta}\ar[d]&\H(K_{i})\ar[d]\ar[r]^{\pi}&\Omega_{i}\ar[r]\ar[d]&0\\
           &0&0&0&}$$
where the existence of the morphism $\mu_{i}$ is from the projectiveness of $B_{i}^{0}$. Let $\sigma_{i}=\left[
              \begin{smallmatrix}
                \H(\overline{f_{i}})&\mu_{i}
              \end{smallmatrix}
            \right]$. Then we have a short exact sequence
            $$0\rightarrow\H(K_{i+1})\oplus B_{i}^{1}\rightarrow \H(M_{i})\oplus B_{i}^{0}\xrightarrow{\sigma_{i}}\H(K_{i})\rightarrow0.$$\par
            Next, we claim that the above sequence keeps exact after applying $\Hom_{B}(N',-)$ for any $N'\in$ add$N$. It suffices to show that the induced map
            $$\Hom_{B}(N',\H(M_{i})\oplus B_{i}^{0})\rightarrow \Hom_{B}(N',\H(K_{i}))$$
            is surjective for $N'\in$ add$Y$ or $N'\in$ add$\H(M)$. If $N'\in$  add$Y$, then we know that $\Hom_{B}(N',\H(K_{i}))=0$ since $N'\in\X(\P)$ and $\H(K_{i})\in \Y(\P)$. Now assume $N'\in$ add$\H(M)$ and $N'$ is a nonzero object. Then there is a module $M'\in$ add$M\bigcap \T(\P)$ such that $N'= \H(M')$. Then we have the following isomorphisms
            \begin{align*}
              \Hom_{B}(N',\H(K_{i})) & =
            \Hom_{B}(\H(M'),\H(K_{i})) \\
               & \cong\Hom_{B}(\H(M'),\H(tK_{i}))\\
               &\cong\Hom_{B}(M',tK_{i})\\
               &\cong\Hom_{B}(M',K_{i})
            \end{align*}
            Thus, for any $g\in \Hom_{B}(N',\H(K_{i}))$, there is a morphism $g':M'\rightarrow K_{i}$ such that $g=\H(g')$.
            Since $\overline{f_{i}}$ is a right add$M$-approximation,   there is a morphism $h:M'\rightarrow M_{i}$ such that $\overline{f_{i}}h=g'$. Then we have the following commutative diagram
            $$\xymatrix{
              &&\H(M')\ar[d]^{\H(g')}\ar[1,-2]_{{\tiny\left[
              \begin{smallmatrix}
                \H(h)\\
                0
              \end{smallmatrix}
            \right]}}\\
             \H(M_{i})\oplus B_{i}^{0}\ar[rr]^-{{\tiny\left[
              \begin{smallmatrix}
                \H(\overline{f_{i}})&\mu_{i}
              \end{smallmatrix}
            \right]}}&&  \H(K_{i})      }.$$
Then the claim holds.\par
 Now, we can construct a long exact sequence
 $$0\rightarrow N_{n}\rightarrow\cdots\rightarrow N_{1}\rightarrow N_{0}\rightarrow \H(U)\rightarrow0$$
 such that it keeps exact after applying $\Hom_{B}(N,-)$, where $N_{n}= \H(M_{n})\oplus B_{n-1}^{1}$, $N_{0}=\H(M_{0})\oplus B_{0}^{0}$ and $N_{i}=\H(M_{i})\oplus B^{0}_{i}\oplus  B^{1}_{i-1}$ for $i=1,2,\cdots,n-1$. Therefore, $\text{Pd}_{\text{End}_{B}(N)}\Hom_{B}(N,\H(U))\leq n$.
\end{proof}
Dually, we have the following lemma.
\begin{lemma}\label{lemma2.12}  Let $A$ be a finite dimensional algebra and  let $\P$ be a 2-term separating silting  complex in {\rm $K^{b}(\text{proj}A)$} such that {\rm Pd$_{A}X\leq1$} for each $X\in\T(\P)$ and {\rm$B=\text{End}_{D^{b}(A)}(\P)$}.  Suppose that $M$ is a cogenerator of {\rm mod$A$} and {\rm$N=DB\oplus\E(M)\oplus Y$} with $Y\in\F(\P)$.  Let {\rm$\Lambda=\text{End}_{B}(N)^{op}$} and {\rm$\Gamma=\text{End}_{A}(M)^{op}$}. Then
{\rm$
  \text{Pd}_{\Lambda}\Hom_{B}(\E(U),N)\leq \text{Pd}_{\Gamma}\Hom_{B}(U,M)
$}
for any $U\in\F(\P)$.
\end{lemma}

\begin{lemma} Let $A$ be a finite dimensional algebra and $\P$ be a 2-term separating silting  complex in {\rm $K^{b}(\text{proj}A)$}  and {\rm$B=\text{End}_{D^{b}(A)}(\P)$}. If $Y$ is a non-projective indecomposable direct summand of {\rm$\text{H}^{0}(\P)$} and $0\rightarrow \tau Y\rightarrow U\rightarrow Y\rightarrow 0$ is an almost split sequence, then $U\in$ {\rm add($\text{H}^{0}(\P)\oplus  \text{H}^{-1}(\nu\P)$)}.
\end{lemma}
\begin{proof} Since $\P$ is  separating, we can assume that $U\cong K\oplus L$ with $K\in\T(\P)$ and $L\in\F(\P)$. The the prove can be divided into two cases.\par
\textbf{Case 1.} Let $K'$ be an indecomposable direct summand of $K$. To prove $K'\in$ add$\text{H}^{0}(\P)$, it suffices to show that $K'$ is Ext-projective in $\T(\P)$ by Lemma \ref{lemma2.14}. If $K'$ is projective, then there is nothing to prove. Assume that $K'$ is not projective. Note that there is an irreducible map $\tau Y\rightarrow K'$. Then there is an irreducible map $\tau K'\rightarrow \tau Y$. Since $\tau\text{H}^{0}(\P)\in \F(\P)$ and $\F(\P)$ is closed under the predecessors, $\tau K'\in\F(\P)$. Hence, $K'$ is Ext-projective in $\T(\P)$.\par
\textbf{Case 2.} Let $L'$ be an indecomposable direct summand of $L$. To prove $L'\in$ add$\text{H}^{-1}(\nu\P))$, it suffices to show that $K'$ is Ext-injective in $\F(\P)$ by Lemma \ref{lemma2.14}.  If $L'$ is injective, then there is nothing to prove. Assume that $L'$ is not injective. Note that there is an irreducible map $ L'\rightarrow Y$. Then there is an irreducible map $Y\rightarrow \tau^{-1} L'$. Since $\text{H}^{0}(\P)\in \T(\P)$ and $\T(\P)$ is closed under the successors, $\tau^{-1} L'\in\T(\P)$. Hence, $L'$ is Ext-injective in $\F(\P)$.
\end{proof}

\begin{lemma}\label{lemma2.17}
Let $A$ be a finite dimensional algebra and $\P$ be a 2-term  silting  complex in {\rm $K^{b}(\text{proj}A)$}  such that {\rm Id$_{A}X\leq1$} for each $X\in\F(\P)$ and {\rm$B=\text{End}_{D^{b}(A)}(\P)$}. Then {\rm$\Hom_{B}(\E(I),\E(X))=0$} where $I\in\F(\P)$ is an injective module, $X\notin$ {\rm add$\text{H}^{-1}(\nu\P)$ and $X\in\F(\P)$ is indecomposable}.
\end{lemma}
\begin{proof}
It suffices to show that $\Hom_{A}(I,X)=0$. Assume that $\Hom_{A}(I,X)\neq0$. Then for any nonzero morphism $u:I\rightarrow X$. It is easy to see that $u$ is not injective. If $u$ is surjective, then there is a short exact sequence $0\rightarrow \Ker u\rightarrow I\rightarrow X\rightarrow0$ in $\F(\P)$. By the assumption on $\F(\P)$, we know that $X$ is injective and hence Ext-injective in $\F(\P)$. It is impossible. Hence, $u$ is not surjective. We consider the short exact $0\rightarrow \Ker u\rightarrow I\rightarrow \Im u\rightarrow0$. Then $\Im u$ is an injective module. Since $\Im u$ is a nonzero injective module, the inclusion map $\Im u\rightarrow X$ is split. It makes a contradiction. Therefore $\Hom_{A}(I,X)=0$.
\end{proof}
\begin{lemma}\label{lemma2.18}Let $A$ be a finite dimensional algebra and $\P$ be a 2-term separating silting  complex in {\rm $K^{b}(\text{proj}A)$}  such that {\rm Id$_{A}X\leq1$} for each $X\in\F(\P)$ and {\rm$B=\text{End}_{D^{b}(A)}(\P)$}. Then {\rm$\Hom_{B}(\E(\tau\text{H}^{0}(\P)),\E(X))=0$} where $X\notin$ {\rm add$\text{H}^{-1}(\nu\P)$ and $X\in\F(\P)$}.
\end{lemma}
\begin{proof} Assume that $X$ is indecomposable. It has shown that $\E(\text{H}^{-1}(\nu\P))$ is an injective $B$-module by Lemma \ref{lemm2-11}. Now assume that $Z$ is is a non-projective indecomposable direct summand of {\rm$\text{H}^{0}(\P)$}. Then there is an AR-sequence
            $$\xymatrix{0\ar[r]&\tau Z\ar[r]^-{{\tiny\left[\begin{smallmatrix}
               f\\g
              \end{smallmatrix}\right]}}&T\oplus Q\ar[r]&Z\ar[r] &0 }.$$
with $T\in\text{add}\text{H}^{0}(\P)$ and $Q\in\text{add}\text{H}^{-1}(\nu\P)$.\par
 Next, we claim that $\E(g):\E(\tau Z)\rightarrow \E(Q)$ is a minimal left almost split sequence. It is easy to check that $\E(g)$ is a left minimal morphism. Next we shall show that $\E(g)$ is a left almost split morphism. Let $W$ be an indecomposable $B$-module and $h:\E(\tau Z)\rightarrow W$ be not split monomorphism. Clearly, $W\in\X(\P)$ and hence $W\cong \E(V)$ for some indecomposable $A$-module $V\in\F(\P)$. Thus, there is a morphism $h':\tau Z\rightarrow V$ such that $\E(h')=h$. Since $\left[\begin{smallmatrix}
               f\\g
              \end{smallmatrix}\right]$ is a left almost split morphism, we have the following commutative diagram
              $$\xymatrix{     &\tau Z\ar[d]_{h'}\ar[r]^-{{\tiny\left[\begin{smallmatrix}
               f\\g
              \end{smallmatrix}\right]}}&T\oplus Q\ar[dl]^{{\tiny\left[\begin{smallmatrix}
               s&t
              \end{smallmatrix}\right]}}\\
              &V&      }$$
              Note that $\Hom_{A}(T,V)=0$ since $T\in\T(\P)$. Thus, $s=0$ and so, $h'=tg$. Then $h=\E(h')=\E(t)\E(g)$.\par
 Now, we focus on $Q$. Note that $\text{H}^{-1}(\nu\P)=\tau \text{H}^{0}(\P)\oplus P$ where $P\in\F(\P)$ is a projective $A$-module. If $P\in$ add$\tau \text{H}^{0}(\P)$, then $Q=\tau Z'$ for some $Z'\in$ add$\tau \text{H}^{0}(\P)$. Suppose that $P\notin$ add$\tau \text{H}^{0}(\P)$. In this case, we claim that $P$ is also an injective $A$-module. Since $\tau^{-1}\text{H}^{-1}(\nu\P)\in$ add$\text{H}^{0}(\P)$. Then $\tau^{-1}P\in\text{H}^{0}(\P)$. If $\tau^{-1}P$ is a nonzero module, then $\tau\tau^{-1}P\cong P\in$ add$\tau \text{H}^{0}(\P)$. It yields a contradiction. Thus, the claim holds. In this case, $Q\cong \tau Z''\oplus I$ where $Z''\in$ add$\text{H}^{0}(\P)$ and $I\in$ add$P$.\par
 We assume that $Q\cong \tau Z''\oplus I$. Now, we can prove that $\Hom_{B}(\E(\tau Z),\E(X))=0$ for  $X\notin$  add$\text{H}^{-1}(\nu\P)$ and $X\in\F(\P)$. Let $u\in\Hom_{B}(\E(\tau Z),\E(X))$. Then $u$ is not spit monomorphism. Hence, we have the following commutative diagram
              $$\xymatrix{     &\E(\tau Z)\ar[d]_{u}\ar[r]^-{{\tiny\E(g)=\left[\begin{smallmatrix}
               \alpha\\\beta
              \end{smallmatrix}\right]}}&\E(\tau Z'')\oplus \E(I)\ar[dl]^{{\tiny\left[\begin{smallmatrix}
               \gamma&\delta
              \end{smallmatrix}\right]}}\\
              &\E(X)&      }$$
              By Lemma \ref{lemma2.17}, $\delta=0$ and so, $\gamma\alpha=u$. Since $\E(\tau Z)\in$ add$\E(\text{H}^{-1}(\nu\P))$, $\E(\tau Z)$ is injective. It follows that $\E(Q)\cong \E(\tau Z)/ \text{Soc}\E(g)$. Then the length $l(\E(Q))$ of $\E(Q)$ smaller than that of $\E(\tau Z)$. Using the induction on the length of $l(\tau Z)$, we know that $\Hom_{B}(\E(\tau Z),\E(X))=0$. It completes the proof.
\end{proof}
Dually, we have the following result.
\begin{lemma}\label{lemm2.16}Let $A$ be a finite dimensional algebra and $\P$ be a 2-term separating silting  complex in {\rm $K^{b}(\text{proj}A)$}  such that {\rm Pd$_{A}X\leq1$} for each $X\in\T(\P)$ and {\rm$B=\text{End}_{D^{b}(A)}(\P)$}. Then {\rm$\Hom_{B}(\H(X),\H(\text{H}^{0}(\P)))=0$} where $X\notin$ {\rm add$\text{H}^{0}(\P)$ and $X\in\T(\P)$}.
\end{lemma}
Now, we are in position to prove our main result.
\begin{theorem}\label{main result} Let $A$ be a finite dimensional algebra and  let $\P$ be a 2-term separating silting  complex in {\rm $K^{b}(\text{proj}A)$} such that {\rm Id$_{A}X\leq1$} for each $X\in\F(\P)$ and {\rm$B=\text{End}_{D^{b}(A)}(\P)$}. Then {\rm rep.dim$B=$rep.dim$A$}. 
\end{theorem}
\begin{proof}  Assume that $\P=P_{-1}\rightarrow P_{0}$. Consider the following exact sequence
\begin{equation*}
  0\rightarrow \text{H}^{-1}(\nu\P)\rightarrow \nu P_{-1}\rightarrow \nu P_{0}\rightarrow \text{H}^{0}(\nu\P)\rightarrow0.
\end{equation*}
By the assumption on $\F(\P)$ and $\text{H}^{-1}(\nu\P)\in\F(\P)$, then $H^{0}(\nu\P)$ is an injective $A$-module. Now, let $M$ be the Auslander generator of $A$ such that $\text{gl.dim}\text{End}_{A}(M)=\text{rep.dim}A=n+2$.\par
Since $\P$ is a separating silting complex, we have $M\cong M_{\T}\oplus M_{\F}$ with $M_{\T}\in$ add$M\bigcap \T(\P)$ and $M_{\F}\in$ add$M\bigcap \F(\P)$. Then by Lemma \ref{lemm2-10}, $H^{0}(\nu\P)\in$ add$M_{\T}$. Now, we set a $B$-module $$N=B\oplus\H(M_{\T})\oplus\E(M_{\F})\oplus\E(H^{-1}(\nu\P)).$$
By Lemma \ref{lemm2-11}, we know that $DB=\E( \text{H}^{-1}(\nu\P))\oplus\H( \text{H}^{0}(\nu\P))$. Thus, $DB\in$ add$N$. It implies that $N$ is a generator and cogenerator of mod$B$. Let $\Lambda=\text{End}_{B}(N)$ and $\Gamma=\text{End}_{A}(M)$. Next, we shall prove that gl.dim$\Lambda\leq$gl.dim$\Gamma$. It suffices to show that Pd$_{\Lambda}\Hom_{B}(N,X)\leq n$ for any indecomposable $B$-module $X$. Since $(\X(\P),\Y(\P))$ is a splitting torsion pair, $X\in\X(\P)$ or $X\in\Y(\P)$. If $X\in\Y(\P)$, then there exists a module $U\in$ $\T(\P)$ such that $X\cong\H(U)$. By Lemma \ref{lemma2.13}, we have Pd$_{\Lambda}\Hom_{B}(N,\H(U))\leq $ Pd$_{\Gamma}\Hom_{A}(M,U)\leq n $. If $X\in\X(\P)$, then there exists a module $V\in$ $\F(\P)$ such that $X\cong\E(V)$. If $V\in$ add$\text{H}^{-1}(\nu\P)$, then $\Hom_{B}(N,\E(V))\in$  add$\Lambda$ and so, Pd$_{\Lambda}\Hom_{B}(N,X)=0\leq n$. Assume that $V\notin$ add$\text{H}^{-1}(\nu\P)$. We assume that $m\leq n$ is the maximal integer such that there is a minimal add$M$-resolution of $V$
$$0\rightarrow M_{m}\rightarrow\cdots\rightarrow  M_{1}\rightarrow M_{0}\rightarrow V\rightarrow0$$
with all $M_{i}\in$ add$M$.
In other words, for each $0\leq i\leq m-1$, we have a short exact sequence
$$0\rightarrow  K_{i+1}\rightarrow M_{i}\xrightarrow{f_{i}} K_{i}\rightarrow0$$
where $f_{i}$ is a right minimal add$M$-approximation, $K_{0}=V$ and $K_{m}=M_{m}$. Since $\F(\P)$ is closed under the predecessors, $K_{i}$ and $M_{i}$ are in $\F(\P)$. Hence, any $K_{i}$ does not lie in add$\text{H}^{-1}(\nu\P)$ since $\text{H}^{-1}(\nu\P)$ is Ext-injective in $\F(\P)$. Applying $\H$ to these short exact sequences, we have
$$0\rightarrow \E(K_{i+1})\rightarrow \E(M_{i})\xrightarrow{\E(f_{i})} \E(K_{i})\rightarrow0$$
for $0\leq i\leq m-1$. \par
Next, we claim that these above induced sequences keep exact after applying $\Hom_{B}(N',-)$ for each indecomposable direct summand $N'$ of $N$. Since $\P$ is splitting, $N'\in \Y(\P)$ or $N'\in \X(\P)$. If $N'\in \Y(\P)$, then by AR-formula, we have $\Ext^{1}_{B}(N', \E(K_{i+1}))=D\overline{\Hom_{B}}(\E(K_{i+1}),\tau N')=0$ for $0\leq i\leq m-1$ since $\tau N'\in \Y(\P)$ and $\E(K_{i+1})\in\X(\P)$. Thus, in this case, the claim holds. Now, we assume that $N'\in\X(\P)$. If $N'\in\E(M_{\F})$, then there exists $N''\in $ add$M_{\F}$ such that $\E(N'')\cong N'$. Then we have the following commutative diagram
$$\xymatrix{
  \Hom_{B}(\E(N''),\E(M_{i}))\ar[d]^{\cong}\ar[r]&\Hom_{B}(\E(N''),\E(K_{i}))\ar[d]^{\cong}&\\
  \Hom_{B}(N'',M_{i})\ar[r]^{\Hom_{B}(N'',f_{i})}&\Hom_{B}(N'',K_{i})\ar[r]&0}$$
Thus, the upper row is surjective and  the claim holds. We consider $N'\in\E(\text{H}^{-1}(\nu\P))$. Note that $\text{H}^{-1}(\nu\P)\cong \tau \text{H}^{0}(\P)\oplus P $ with $P\in$ add$A\bigcap\F(\P)\subseteq$ add$M_{\F}$. If
$N'\in$ add$\E(P)$, then there is noting to prove. We assume that $N'\in$ add$\E(\tau \text{H}^{0}(\P))$. By Lemma \ref{lemma2.18}, we know that $\Hom_{B}(N',\E(K_{i}))=0$. Hence, the claim holds.\par
Therefore, there is a long exact sequence
$$0\rightarrow \Hom_{B}(N,\E(M_{m}))\rightarrow\cdots\rightarrow \Hom_{B}(N,\E(M_{0}))\rightarrow \Hom_{B}(N,\E(V))\rightarrow0$$
where all $\E(M_{i})\in$ add$N$. This implies that Pd$_{\Lambda}\Hom_{B}(N,\E(V))\leq m\leq n$ and so Pd$_{\Lambda}\Hom_{B}(N,X)\leq n$ for any $X\in$ mod$B$. Then, gl.dim$\text{End}_{B}(N)\leq n+2=$gl.dim$\text{End}_{A}(M)=$rep.dim$A$. Therefore, rep.dim$B\leq$rep.dim$A$.\par
Assume that $\Theta$ is an Aulsander generator of mod$B$ and the induced silting complex $\Q=Q_{-1}\xrightarrow{\sigma}Q_{0}$ with $Q_{i}\in$ proj$B$. Since $\Q$ is a separating silting complex, we can write $\Theta=\Theta_{\T}\oplus \Theta_{\F}$ where $\Theta_{\T}\in$ add$\Theta\bigcap\T(\Q)$ and $\Theta_{\F}\in$ add$\Theta\bigcap\F(\Q)$. We consider the following exact sequence in mod$B$
$$0\rightarrow \text{H}^{-1}(\Q)\rightarrow Q_{-1}\xrightarrow{\sigma}Q_{0}\rightarrow \text{H}^{0}(\Q)\rightarrow 0.$$
Since $\text{H}^{0}(\Q)\in\T(\Q)=\X(\P)$, by Proposition \ref{prop2.12}, we have Pd$_{B}\text{H}^{0}(\Q)\leq1$ and so, $\text{H}^{-1}(\Q)$ is a projective $B$-module. By \cite[Proposition 5.4]{Hoshino}, $\text{H}^{-1}(\Q)\in$ add$\Theta_{\F}$. Now, we define an $A$-module $$\Theta'=DB\oplus\E(\Theta_{\F})\oplus\H(\Theta_{\T})\oplus\H(\text{H}^{0}(\Q)).$$
By Lemma \ref{lem1-4}, we know that $A\cong\text{End}_{B}(\Q)=\H(\text{H}^{0}(\Q))\oplus\E(\text{H}^{-1}(\Q))\in$ add$\Theta'$. Hence, $\Theta'$ is a generator and cogenerator of mod$A$.
Dually, we apply Lemma \ref{lemma2.0}(3), Proposition \ref{proposition 2.10}, Lemma \ref{lemma2.12} and  Lemma \ref{lemm2.16} to the 2-term silting complex $\Q$ in $K^{b}(\text{proj}B)$ and $\Theta'$. Then we obtain rep.dim $A$=rep.dim $\text{End}_{B}(\Q)\leq$ rep.dim $B$. Therefore, rep.dim $A$= rep.dim $B$.
\end{proof}
The following result is directly from Theorem \ref{main result}.
\begin{corollary}  Let $A$ be a finite dimensional hereditary algebra. If $\P$ is a 2-term separating  stilting complex with {\rm$B=\text{End}_{D^{b}(A)}(\P)$}, then {\rm rep.dim $B\leq $3}. Moreover, if $A$ is representation-finite, then $B$ is also representation-finite.
\end{corollary}

Let  $T$ be a  classical tilting right $A$-module. If one take $\P$ as the projective resolution of $T$, then $\P$ is a 2-term silting complex such that $\T(\P)=\T(T)=\text{Ker} \Ext_{A}^{1}(T,-)$ and $\F(\P)=\F(T)=\text{Ker} \Hom_{A}(T,-)$. Moreover, one can see that ($\X(\P)$, $\Y(\P)$)=($\X(\T)$, $\Y(\T)$) in mod$B$, where  $B=\text{End}_{A}(T)=\text{End}_{D^{b}(A)}(\P)$. It is easy to prove that $\P$ is separating when $T$ is separating. If $T$ is splitting, then Id$_{A}X\leq1$ for any $X\in\F(\T)=\F(\P)$. Thus, we have the following consequence.
\begin{corollary}\label{cor1} {\rm\cite[Theorem 3.1]{ChenHu}}  Let $A$ be a finite dimensional algebra and {\rm$B=\text{End}_{A}(T)$}. If $T$ is a separating and splitting  tilting $A$-module, then {\rm rep.dim $A$= rep.dim $B$}.
\end{corollary}
\section{Examples}
Firstly, we would like to give two examples to illustrate that  anyone of the conditions of Theorem \ref{main result} cannot be removed.
\begin{example} Let $A$ be a path algebra over a filed $k$ and its quiver $Q_{A}$ given by
$$\xymatrix{&\ar[dl]2& \\
1&&4\ar[ul]\ar[dl] \\
&\ar[ul]3&           }$$
Set a right $A$-module $T=1\oplus{\begin{array}{c}
                  4 \\
                  3 \\
                  1 \\
                \end{array}}\oplus{\begin{array}{c}
                  4 \\
                  2 \\
                  1 \\
                \end{array}}\oplus 4$. It is easy to see that $T$ is a tilting module. Let $\P$ be the projective resolution of $T$. Then $\P$ is a 2-term silting complex. Since $A$ is a hereditary algebra, $\P$ is a splitting silting complex. Then $B=\text{End}_{D^{b}(A)}(\P)=\text{End}_{A}(T)$ is given by the quiver $Q_{B}$
              $$\xymatrix{&\ar[dl]_{\beta}2& \\
1&&4\ar[ul]_{\alpha}\ar[dl]^{\gamma} \\
&\ar[ul]^{\delta}3&           }$$
with relations $\alpha\beta=0$ and $\gamma\delta=0$. The AR-quiver of mod$B$ is given by
\begin{equation}\label{eqd3-1}
  \begin{split}
\xymatrix@C=1.5em@R=1.5em{
  &*-<1.2mm>{\text{{
     $\begin{array}{ccc}
        2\\1
     \end{array}$}}}\ar[dr]\ar@{.}[rr]
  &&*-<1mm>{\text{{
     $\begin{array}{ccc}
        3
     \end{array}$}}}\ar[dr]\ar@{.}[rr]&&*-<1mm>{\text{{
     $\begin{array}{ccc}
        4\\2
     \end{array}$}}}\ar[dr]& \\
    *-<1mm>{ \text{{
     $\begin{array}{ccc}
        1
     \end{array}$}}}\ar[ur]\ar[dr]\ar@{.}[rr]&&*-<1mm>{\text{{
     $\begin{array}{ccc}
        2~3\\1
     \end{array}$}}}\ar[ur]\ar[dr]&&*-<1mm>{\text{{
     $\begin{array}{ccc}
        4\\2~3
     \end{array}$}}}\ar[ur]\ar[dr]\ar@{.}[rr]&&*-<1mm>{\text{{
     $\begin{array}{ccc}
        4
     \end{array}$}}}  \\
     &*-<1mm>{\text{{
     $\begin{array}{ccc}
        3\\1
     \end{array}$}}}\ar[ur]\ar@{.}[rr]&&*-<1mm>{\text{{
     $\begin{array}{ccc}
        2
     \end{array}$}}}\ar[ur]\ar@{.}[rr]&&*-<1mm>{\text{{
     $\begin{array}{ccc}
        4\\3
     \end{array}$}}}\ar[ur]&  }
\end{split}
\end{equation}
where the dot lines are the $\tau$-orbits.
     The induced splitting torsion pair ($\X(\P)$, $\Y(\P)$) in mod$B$ is given by
     \begin{align*}
       \X(\P) &= \{\text{{
     $\begin{array}{ccc}
        4\\2
     \end{array}$}}, \text{{
     $\begin{array}{ccc}
        4\\3
     \end{array}$}},\text{{
     $\begin{array}{ccc}
        4
     \end{array}$}} \}\\
       \Y(\P) &= \{\text{{
     $\begin{array}{ccc}
        1
     \end{array}$}},\text{{
     $\begin{array}{ccc}
        2\\1
     \end{array}$}},\text{{
     $\begin{array}{ccc}
        3\\1
     \end{array}$}},\text{{
     $\begin{array}{ccc}
        2~ 3\\1
     \end{array}$}},\text{{
     $\begin{array}{ccc}
        3
     \end{array}$}}, \text{{
     $\begin{array}{ccc}
        2
     \end{array}$}},\text{{
     $\begin{array}{ccc}
     4\\2~3
     \end{array}$}}\}
     \end{align*}
     Note that pd$_{B}S(4)=2>1$ where $S(4)\in\X(\P)$ is a simple $B$-module corresponding to the vertex 4 of $Q_{B}$. Thus, by Lemma \ref{prop2.12}, $\P$ is not a separating silting complex. It is easy to see that rep.dim$B=2<$rep.dim$A=3$.
\end{example}
The following example given by \cite{Abe}. Due to this example, we can see that the homological dimension restriction on $\F(\P)$ cannot be removed.
\begin{example} Let $A$ be a $k$-algebra given by the following quiver
              $$\xymatrix{&\ar[dl]_{\beta}2& \\
1&&4\ar[ul]_{\alpha}\ar[dl]^{\gamma} \\
&\ar[ul]^{\delta}3&           }$$
with relations $\alpha\beta=0$ and $\gamma\delta=0$. Then the AR-quiver of mod$A$ is form as the graph (\ref{eqd3-1}). There is a splitting torsion pair ($\T$, $\F$) in mod$A$
\begin{align*}
       \T &= \{\text{{
     $\begin{array}{ccc}
        4\\2~3
     \end{array}$}},\text{{
     $\begin{array}{ccc}
        4\\2
     \end{array}$}}, \text{{
     $\begin{array}{ccc}
        4\\3
     \end{array}$}},\text{{
     $\begin{array}{ccc}
        4
     \end{array}$}} \}\\
       \F &= \{\text{{
     $\begin{array}{ccc}
        1
     \end{array}$}},\text{{
     $\begin{array}{ccc}
        2\\1
     \end{array}$}},\text{{
     $\begin{array}{ccc}
        3\\1
     \end{array}$}},\text{{
     $\begin{array}{ccc}
        2~3\\1
     \end{array}$}},\text{{
     $\begin{array}{ccc}
        3
     \end{array}$}}, \text{{
     $\begin{array}{ccc}
        2
     \end{array}$}}\}.
     \end{align*}
The 2-term silting complex $\P$ is given by the direct sums of the following complexes in $K^{b}(\text{proj}A)$
\begin{align*}
  \P_{1} &=\text{{
     $\begin{array}{ccc}
       0
     \end{array}$}}\rightarrow  \text{{
     $\begin{array}{ccc}
        4\\2~3
     \end{array}$}}\\
  \P_{2} &=\text{{
     $\begin{array}{ccc}
        3\\1
     \end{array}$}}\rightarrow \text{{
     $\begin{array}{ccc}
        4\\2~3
     \end{array}$}}\\
     \P_{3} &=\text{{
     $\begin{array}{ccc}
        2\\1
     \end{array}$}}\rightarrow \text{{
     $\begin{array}{ccc}
        4\\2~3
     \end{array}$}}\\
  \P_{4} &=  \text{{
     $\begin{array}{ccc}
       1
     \end{array}$}}\rightarrow  \text{{
     $\begin{array}{ccc}
        0
     \end{array}$}}.
\end{align*}
In this case, $\T=\T(\P)$ and $\F=\F(\P)$. Thus, the silting complex $\P$ is separating. Moreover, the endomorphism algebra $B=\End_{D^{b}(A)}(\P)$ is a path algebra given by the following quiver, whose underlying graph is a Euclidean graph of $\widetilde{A_{3}}$

$$\xymatrix{&\ar[dl]2\ar[dr]& \\
1&&4 \\
&\ar[ul]3\ar[ur]&           }$$
Thus,  $\End_{D^{b}(A)}(\P)$ is representation-infinite. Thus, rep.dim$A=2<$rep.dim$B=3$. Note that id$_{A}S(1)=2$ where $S(1)\in\F(\P)$  is a simple $A$-module corresponding to the vertex 1. Therefore, the separating silting complex $\P$ does not satisfy the homological dimension restriction on $\F(\P)$.
\end{example}
Secondly, we should illustrate that Theorem \ref{main result}  gives a proper generalization of the result in \cite{ChenHu}. It means that one can find a 2-term silting complex $\P$ satisfying the assumption in Theorem \ref{main result}, but it cannot be induced by a tilting module. In fact, this silting complex  is not rare.
\begin{example}\label{ex4.3} Let $A$ be a path algebra given by the following quiver
$$\xymatrix{3\ar[r] &2\ar[r] &1   }$$
The AR-quiver is given by
$$\xymatrix@C=1.8em@R=1.8em{*+<1mm>{\text{{
     $\begin{array}{ccc}
1\\
     \end{array}$}}}\ar[dr]&&*+<1mm>{\text{{
     $\begin{array}{ccc}
       2\\
     \end{array}$}}}\ar[dr]&&*+<1mm>{\text{{
     $\begin{array}{ccc}
       3\\
     \end{array}$}}} \\
&*-<1mm>{\text{{
     $\begin{array}{ccc}
       2\\1
     \end{array}$}}}\ar[ur]\ar[dr]&&*-<1mm>{\text{{
     $\begin{array}{ccc}
       3\\2
     \end{array}$}}}\ar[ur]& \\
     &&*-<1mm>{\text{{
     $\begin{array}{ccc}
       3\\2\\1
     \end{array}$}}}\ar[ur]&&}$$
\end{example}
Then one can take a splitting torsion pair ($\T$, $\F$) as follows.
\begin{align*}
  \T &=\{\text{{
     $\begin{array}{ccc}
       2
     \end{array}$}}, \text{{
     $\begin{array}{ccc}
       3\\2
     \end{array}$}},\text{{
     $\begin{array}{ccc}
       3
     \end{array}$}}\} \\
 \F &=\{\text{{
     $\begin{array}{ccc}
1\\
     \end{array}$}}, \text{{
     $\begin{array}{ccc}
2\\1
     \end{array}$}},\text{{
     $\begin{array}{ccc}
3\\2\\1
     \end{array}$}}\}
\end{align*}
Note that this torsion pair cannot be induced by tilting modules in mod$A$. Let $\P$ be a 2-term complex given by the direct sums of the following complex in $K^{b}(\text{proj}A)$
\begin{align*}
  \P_{1}=\text{{
     $\begin{array}{ccc}
       1
     \end{array}$}}\rightarrow \text{{
     $\begin{array}{ccc}
       2\\1
     \end{array}$}} && \P_{2}=\text{{
     $\begin{array}{ccc}
       1
     \end{array}$}}\rightarrow \text{{
     $\begin{array}{ccc}
       3\\2\\1
     \end{array}$}} && \P_{3}=\text{{
     $\begin{array}{ccc}
       3\\2\\1
     \end{array}$}}\rightarrow\text{{
     $\begin{array}{ccc}
       0
     \end{array}$}}
\end{align*}
It is easy to check that $\P$ is a silting complex such that $\T=\T(\P)$ and $\F=\F(\P)$. Thus, $\P$ is a separating silting complex. It is not difficult to see that $\P$ is not a projective resolution of a tilting module. Since $A$ is a hereditary algebra, by Theorem \ref{main result}, $\text{End}_{D^{b}(A)}(\P)$ is also representation-finite.
\section{The representation dimension of {\rm$\End_{A}(\text{H}^{0}(\P))$}}
Let $A$ be a finite dimensional algebra and $\P$ be a silting  complex in {\rm $K^{b}(\text{proj}A)$}. We set $\mathfrak{A}=\text{ann}_{A}(\P)$ the annihilator of $\text{H}^{0}(\P)$. It is well-known that $\text{H}^{0}(\P)$ is a tilting $A/\mathfrak{A}$-module see \cite[Lemma 3.1]{Abe}. We denote by $\text{Fac}M$ is a full subcategory consisting of all factor modules of the finite copies of a $A$-module $M$, i.e $$\text{Fac}M=\{~X\in \text{mod}\Lambda~|~M^{(n)}\twoheadrightarrow X~\text{for some integer}~n~\}.$$
\begin{proposition}\label{pro5.1}Let $A$ be a finite dimensional algebra and $\P$ be a 2-term  splitting and separating silting  complex in {\rm $K^{b}(\text{proj}A)$}. Then {\rm$\text{H}^{0}(\P)$} is a splitting and separating tilting $A/\mathfrak{A}$-module.
\end{proposition}
\begin{proof} Firstly, we prove that $\text{H}^{0}(\P)$ is a splitting  tilting $A/\mathfrak{A}$-module.\par
 The canonical full embedding $\text{mod}A/\mathfrak{A}\hookrightarrow \text{mod}A$ induces $\text{Fac}(\text{H}^{0}(\P)_{A/\mathfrak{A}})=\text{Fac}(\text{H}^{0}(\P)_{A})$. Since $\text{H}^{0}(\P)$ is a tilting $A/\mathfrak{A}$-module, $\T(\text{H}^{0}(\P)_{A/\mathfrak{A}})=\text{Fac}(\text{H}^{0}(\P)_{A/\mathfrak{A}})$. It is well-known that $\T(\P)=\text{Fac}(\text{H}^{0}(\P)_{A})$. Hence, $\T(\text{H}^{0}(\P)_{A/\mathfrak{A}})=\T(\P)$. For any indecomposable $Y\in$ mod$A/\mathfrak{A}$. There is a short exact sequence in mod$A/\mathfrak{A}$
\begin{equation}\label{eq5.1}
  0\rightarrow Y\rightarrow \text{E}(Y)\rightarrow \Omega^{-1}_{A/\mathfrak{A}}(Y)\rightarrow0
\end{equation}
where $\text{E}(Y)$ is the injective hull of $Y$ and $\Omega^{-1}_{A/\mathfrak{A}}(Y)$ is the 1-th cosyzygy of $Y$. Note that $D(A/\mathfrak{A})\in\T(\text{H}^{0}(\P)_{A/\mathfrak{A}})$. Then, $\Omega^{-1}_{A/\mathfrak{A}}(Y)\in\T(\text{H}^{0}(\P)_{A/\mathfrak{A}})$ and so $\Omega^{-1}_{A/\mathfrak{A}}(Y)\in\T(\P)$ as a right $A$-module. \par
Assume that $Y\in\F(\text{H}^{0}(\P)_{A/\mathfrak{A}})$. We claim that id$_{A/\mathfrak{A}}Y\leq1$.  It suffices to prove that $\Ext^{2}_{A/\mathfrak{A}}(X,Y)=0$ for any indecomposable $A/\mathfrak{A}$-module $X$.\par
 Since $\P$ is separating, ($\T(\P)$, $\F(\P)$) is splitting. Note that $X$ is also an indecomposable $A$-module. Then $X$ is either in $\F(\P)$ or $\T(\P)$. \par
 Assume that $X\in\F(\P)$.  Applying $\Hom_{A/\mathfrak{A}}(X,-)$ to the sequence (\ref{eq5.1}), we have the exact sequence
$$\Ext_{A/\mathfrak{A}}^{1}(X,\Omega^{-1}_{A/\mathfrak{A}}(Y))\rightarrow
\Ext_{A/\mathfrak{A}}^{2}(X,Y)\rightarrow\Ext_{A/\mathfrak{A}}^{2}(X,\text{E}(Y))$$
Note that there is a natural monoic map $\Ext_{A/\mathfrak{A}}^{1}(X,\Omega^{-1}_{A/\mathfrak{A}}(Y))\hookrightarrow\Ext_{A}^{1}(X,\Omega^{-1}_{A/\mathfrak{A}}(Y))$ see \cite{Oort}. Since $\Omega^{-1}_{A/\mathfrak{A}}(Y)\in\T(\P)$, we have $\Ext_{A}^{1}(X,\Omega^{-1}_{A/\mathfrak{A}}(Y))=0$ and hence, $\Ext_{A/\mathfrak{A}}^{1}(X,\Omega^{-1}_{A/\mathfrak{A}}(Y))$ $=0$. It implies that $\Ext_{A/\mathfrak{A}}^{2}(X,Y)=0$.\par
 Now we assume $X\in\T(\P)$.
By \cite[Lemma 2.13(2)]{Abe}, we have
\begin{align*}
  \Hom_{D^{b}(A)}(\P,Y)&\cong\Hom_{A}(\text{H}^{0}(\P),Y)  \\
   &\cong \Hom_{A/\mathfrak{A}}(\text{H}^{0}(\P),Y)\\
   &=0.
\end{align*}
Thus, $Y\in\F(\P)$ as right $A$-module. Note that there is a natural embedding $\Ext_{A/\mathfrak{A}}^{2}(X,Y)\hookrightarrow\Ext_{A}^{2}(X,Y)$ induced by the canonical full embedding $\text{mod}A/\mathfrak{A}\hookrightarrow \text{mod}A$. Since $\P$ is splitting, we have $\Ext_{A}^{2}(X,Y)=0$. Thus, $\Ext_{A/\mathfrak{A}}^{2}(X,Y)=0$.\par
Next, we claim that $\text{H}^{0}(\P)$ is a separating  tilting $A/\mathfrak{A}$-module.\par
Indeed, for any pair $X\in\F(\text{H}^{0}(\P)_{A/\mathfrak{A}})$, $Y\in \T(\text{H}^{0}(\P)_{A/\mathfrak{A}})$, it is easy to see that $X\in\F(\P)$ and $Y\in \T(\P)$. Note that there is a monoic map $\Ext_{A/\mathfrak{A}}^{1}(X,Y)\hookrightarrow \Ext_{A}^{1}(X,Y)$. Since $\P$ is separating, $\Ext_{A}^{1}(X,Y)=0$ and so $\Ext_{A/\mathfrak{A}}^{1}(X,Y)=0$. Then, the claim holds.
\end{proof}
By Corollary \ref{cor1} and Proposition \ref{pro5.1}, one can get the following consequence.
\begin{corollary}Let $A$ be a finite dimensional algebra and $\P$ be a 2-term  splitting and separating silting  complex in {\rm $K^{b}(\text{proj}A)$}. Then {\rm rep.dim$\End_{A}(\text{H}^{0}(\P))=$ rep.dim$A/\mathfrak{A}$}.
\end{corollary}

\bibliographystyle{amsplain}

\end{document}